\documentclass[12pt,a4paper]{amsart}%

\usepackage{amsfonts,amsmath,amssymb}
\usepackage{amssymb,amsthm,amsxtra}
\usepackage[usenames]{color}
\usepackage{amscd}
\usepackage{amsthm}
\usepackage{amsfonts}
\usepackage{amssymb}
\usepackage{amsmath}
\usepackage{graphicx}
\usepackage{hyperref}
\usepackage{enumerate}%
\usepackage[all]{xy}
\usepackage[usenames,dvipsnames]{xcolor}
\usepackage{mathrsfs}
\usepackage[capitalize]{cleveref}
\usepackage{cleveref}
 \newtheorem*{corollary*}{Corollary}
 \newtheorem*{construction*}{Construction}
 \newtheorem*{definition*}{Definition}
 \newtheorem*{notation*}{Notation}
 \newtheorem*{lemma*}{Lemma}
 \newtheorem*{theorem*}{Theorem}
 \newtheorem*{remark*}{Remark}
 \newtheorem*{example*}{Example}
 \newtheorem*{conjecture*}{Conjecture}
 \newtheorem*{condition*}{Condition}
 \newtheorem*{result*}{Result}
 \newtheorem*{property*}{Property}

 \newtheorem*{cor*}{Corollary}
 \newtheorem*{const*}{Construction}
 \newtheorem*{defn*}{Definition}
 \newtheorem*{notn*}{Notation}
 \newtheorem*{lem*}{Lemma}
 \newtheorem*{thm*}{Theorem}
 \newtheorem*{rem*}{Remark}
 \newtheorem*{exm*}{Example}
 \newtheorem*{conj*}{Conjecture}

 \newtheorem{lemma}{Lemma}[subsection]
 
 \newtheorem{remark}[lemma]{Remark}

 \newtheorem{thm}[lemma]{Theorem}
 \newtheorem{prop}[lemma]{Proposition}
 \newtheorem{lem}[lemma]{Lemma}
 \newtheorem{defn}[lemma]{Definition}
 \newtheorem{notn}[lemma]{Notation}
 \newtheorem{cor}[lemma]{Corollary}

 \newtheorem{introtheorem}{Theorem}

 \crefname{introtheorem}{theorem}{theorems}
 \Crefname{introtheorem}{Theorem}{Theorems}
  \newtheorem{introthm}[introtheorem]{Theorem}
   \crefname{introthm}{theorem}{theorems}
 \Crefname{introthm}{Theorem}{Theorems}

  \crefname{introcorollary}{corollary}{corollaries}
 \Crefname{introcorollary}{Corollary}{Corollaries}
 
  \newtheorem{introquest}[introtheorem]{Question}
 \crefname{introquest}{Question}{Questions}
 \Crefname{introquest}{Question}{Questions}

 \newtheorem{introcor}[introtheorem]{Corollary}
   \crefname{introcor}{corollary}{corollaries}
 \Crefname{introcor}{Corollary}{Corollaries}
 
   \crefname{introconjecture}{conjectures}{conjectures}
 \Crefname{introconjecture}{Conjecture}{Conjectures}
 
    \crefname{introconj}{conjectures}{conjectures}
 \Crefname{introconj}{Conjecture}{Conjectures}

     \crefname{introlem}{lemma}{lemmas}
 \Crefname{introlem}{Lemma}{Lemmas}
 
 \crefname{introremark}{remark}{remarks}
 \Crefname{introremark}{Remark}{Remarks}
 
  \crefname{introrem}{remark}{remarks}
 \Crefname{introrem}{Remark}{Remarks}

   \crefname{introprop}{Proposition}{Propositions}
 \Crefname{introprop}{Proposition}{Propositions}

   \crefname{introdefn}{definition}{definitions}
 \Crefname{introdefn}{Definition}{Definitions}
 
   \crefname{intronotn}{notation}{notations}
 \Crefname{intronotn}{Notation}{Notations}
 
   \crefname{introtask}{task}{tasks}
 \Crefname{introtask}{Task}{Tasks}
 
  \crefname{introprob}{problem}{problems}
 \Crefname{introprob}{Problem}{Problems}
 
   \crefname{introquestion}{question}{questions}
 \Crefname{introquestion}{Question}{Questions}

   \crefname{introexm}{example}{example}
 \Crefname{introquestion}{Example}{Example}
 \crefname{theorem}{theorem}{theorems}
 \Crefname{theorem}{Theorem}{Theorems}
  \crefname{thm}{theorem}{theorems}
 \Crefname{thm}{Theorem}{Theorems}
  \crefname{corollary}{Corollary}{Corollaries}
 \Crefname{corollary}{Corollary}{Corollaries}
   \crefname{cor}{Corollary}{Corollaries}
 \Crefname{cor}{Corollary}{Corollaries}
   \crefname{conjecture}{conjectures}{conjectures}
 \Crefname{conjecture}{Conjecture}{Conjectures}

    \crefname{conj}{conjectures}{conjectures}
 \Crefname{conj}{Conjecture}{Conjectures}

     \crefname{lem}{lemma}{lemmas}
 \Crefname{lem}{Lemma}{Lemmas}

      \crefname{lemma}{Lemma}{Lemmas}
 \Crefname{lemma}{Lemma}{Lemmas}

 \crefname{remark}{remark}{remarks}
 \Crefname{remark}{Remark}{Remarks}

  \crefname{rem}{remark}{remarks}
 \Crefname{rem}{Remark}{Remarks}

   \crefname{rem}{remark}{remarks}
 \Crefname{rem}{Remark}{Remarks}

   \crefname{proposition}{Proposition}{Proposition}
 \Crefname{proposition}{Proposition}{Proposition}

    \crefname{prop}{Proposition}{Propositions}
 \Crefname{prop}{Proposition}{Propositions}

   \crefname{defn}{definition}{definitions}
 \Crefname{defn}{Definition}{Definitions}

   \crefname{notn}{notation}{notations}
 \Crefname{notn}{Notation}{Notations}

   \crefname{task}{task}{tasks}
 \Crefname{task}{Task}{Tasks}

  \crefname{prob}{problem}{problems}
 \Crefname{prob}{Problem}{Problems}

   \crefname{question}{question}{questions}
 \Crefname{question}{Question}{Questions}

\newcommand{\alp}{\alpha}

\newcommand{\Id}{\operatorname{Id}}

\newcommand{\Ker}{\operatorname{Ker}}

\newcommand{\ad}{\operatorname{ad}}

\newcommand{\C}{\mathbb{C}}
\newcommand{\bN}{\mathbb{N}}

\newcommand{\bfG}{\mathbf{G}}
\newcommand{\bfH}{\mathbf{H}}

\newcommand{\bfY}{\mathbf{Y}}
\newcommand{\bfX}{\mathbf{X}}
\newcommand{\bfO}{\mathbf{O}}

\newcommand{\bfP}{\mathbf{P}}

\newcommand{\R}{\mathbb{R}}

\newcommand{\sdim}{\operatorname{sdim}}

\newcommand{\onto}{\twoheadrightarrow}

\providecommand{\fg}{\mathfrak{g}}
\providecommand{\fh}{\mathfrak{h}}

\providecommand{\cN}{\mathcal{N}}

\providecommand{\sub}{\subset}

\newcommand{\Dima}[1]{{{#1}}}
\newcommand{\DimaA}[1]{{{#1}}}
\newcommand{\DimaB}[1]{{{#1}}}

\begin{document}

\title{Symplectic complexity of reductive group actions}
\author{Avraham Aizenbud}
\address{Faculty of Mathematics and Computer Science, Weizmann
Institute of Science, POB 26, Rehovot 76100, Israel }
\email{aizenr@gmail.com}
\urladdr{http://www.aizenbud.org}

\author{Dmitry Gourevitch}
\email{dimagur@weizmann.ac.il}
 \urladdr{http://www.wisdom.weizmann.ac.il/~dimagur}

\keywords{Algebraic group, nilpotent orbit, homogeneous space, spherical space, spherical space, non-commutative harmonic analysis, isotropic subvariety.}
\subjclass[2020]{20G07, 14L30, 53D20, 53D12, 43A85 20G05,14L24}
%
%
%
%
%
%
%
%
\date{\today}

\maketitle

\begin{abstract}
Let a complex algebraic reductive group $\bf G$ act on a  complex algebraic manifold $\bf X$. For a $\bfG$-invariant  subvariety $\Xi$ of the nilpotent cone $\cN(\fg^*)\sub \fg^*$ we define a notion of $\Xi$-symplectic complexity of $\bfX$. This notion generalizes the notion of complexity defined in \cite{Vin}.
We prove several properties of this notion, and relate it to the notion of $\Xi$-complexity defined in \cite{AG} motivated by its relation with representation theory. 
\end{abstract}
%
%


\section{Introduction}

Let $\bf G$ be a complex algebraic reductive group and $\bf X$ be a complex algebraic  $\bf G$ -manifold. In \cite{Vin}, Vinberg defined a notion of complexity  of $\bf X$ that measures how far $\bf X$ is from being spherical\footnote{In fact, \cite{Vin} does not assume that $\bfX$ is smooth.}. For a $\bfG$-invariant  subvariety $\Xi$ of the nilpotent cone $\cN(\fg^*)\sub \fg^*$, we defined in \cite{AG} the a notion of $\Xi$-spherical varieties, and $\Xi$-complexity of $\bf X$, denoted $c_{\Xi}(\bfX)$, 
that measures how far $\bf X$ is from being $\Xi$-spherical. 

In this paper we define a variant of this notion that we call \emph{$\Xi$-symplectic complexity}, and denote $sc_{\Xi}(\bfX)$. This notion seems to be more coarse, but better behaved, and equally good for applications in representation theory. In particular, we prove the following theorem.

\begin{introthm}[\S \ref{sec:sc}]\label{thm:main}Let $\Xi\sub \cN(\fg^*)$ be a $\bfG$-invariant (locally closed) subvariety. 
\begin{enumerate}[(i)]
\item \label{it:geq} $sc_{\Xi}(\bfX)\geq \max(c_{\Xi}(\bfX),0)$
\item \label{it:max} Let $\{\bfY_i\}_{i=1}^k$ be smooth $\bfG$-invariant (locally closed) subvarieties of $\bfX$ such that $\bfX = \bigcup_i \bfY_i$, and $\{\Xi_i\}_{i=1}^l$ be smooth $\bfG$-invariant (locally closed) subvarieties of $\Xi$ such that $\Xi = \bigcup_j \Xi_j$.
 Then $sc_{\Xi}(\bfX)=\max_{i,j} sc_{\Xi_j}(\bfY_i)$.

\item \label{it:mod} For any parabolic subgroup $\bf P\sub G$, $sc_{\bf \overline{O_P}}(\bfX)$ equals the $\bfP$-modality\footnote{See \S \ref{subsec:mod} below.} of $\bfX$,  where $\bf O_P$ denotes the Richardson orbit defined by $\bfP$.
\item \label{it:N} $sc_{\cN(\fg^*)}(\bfX)=c_{\cN(\fg^*)}(\bfX)=c(\bfX)$.
\end{enumerate}
\end{introthm}

We are mainly interested in the case when $\bfG$ has finitely many orbits on $\bfX$. By Theorem \ref{thm:main}\eqref{it:max} the computation of $sc_{\Xi}(\bfX)$ for such $\bfX$ reduces to the case when both $\bf X$ and $\Xi$ are $\bfG$-transitive. In this case we give a simpler description for $sc_{\Xi}(\bfX)$ in Proposition \ref{prop:tran} below.

The equality $c_{\cN(\fg^*)}(\bfX)=c(\bfX)$ was proven in \cite{AG}. For this reason, the varieties with $c_{\Xi}(\bfX)\leq0$ are called $\Xi$-spherical. The study of this notion is motivated by \cite[Theorem D]{AG}, that implies that  for closed $\Xi$ and $\Xi$-spherical $\bfX$, the Casselman-Wallach representations of $\bfG(\C)$ from a certain category defined by $\Xi$ appear in  the space of Schwartz functions on $\bfX(\C)$ with  finite multiplicities\footnote{ If $\bfG$ and $\bfX$ are defined over $\R$ then the same holds for their real points.}.

Theorem \ref{thm:main}, together with \cite[Theorem B]{AG}, implies the following corollary. 

\begin{introcor}[\S \ref{sec:sc}]\label{cor:main}
For any parabolic subgroup $\bf P\sub G$, the following are equivalent.
\begin{enumerate}[(a)]
\item \label{it:sc0} $sc_{\bf \overline{O_P}}(\bfX)=0$.
\item \label{it:Ospher} $\bfX$ is ${\bf \overline{O_P}}$-spherical.
\item \label{it:Pfin} $\bfP$ has finitely many orbits on $\bfX$.
\end{enumerate}
\end{introcor}

If $0\notin \Xi$ then $c_{\Xi}(\bfX)$ may be negative. However, we do not have an example in which $c_{\Xi}(\bfX)$ is non-negative but different from $sc(\bfX)$.  This leads to the following question.
\begin{introquest}$\,$
\begin{enumerate}[(i)]
\item Do we have $sc_{\Xi}(X)= \max(c_{\Xi}(\bfX),0)$? In particular, does $\Xi$-spherical imply $sc_{\Xi}(\bfX)=0$?

\item For a parabolic subgroup $\bf P\sub G$, does $c_{\bf \overline{O_P}}(\bfX)$ equal $sc_{\bf \overline{O_P}}(\bfX)$? 

\item \Dima{Let $\bf H\sub G$ be an algebraic subgroup, let $\fh$ denote the Lie algebra of $\bfH$ and $\fh^{\bot}\subset \fg^*$denote its annihilator. Let $\bfO\subset \cN(\fg^*)$ be a nilpotent orbit that intersects $\fh^{\bot}$. Do we have $sc_{\bfO}({\bf G/H})=sc_{\overline{\bfO}}({\bf G/H})$?}
\end{enumerate}
\end{introquest}

\Dima{
\subsection{Motivation}
While the notion of $\Xi$-complexity $c_{\Xi}(\bfX)$ is useful in representation theory, it seems not to be robust enough. For example, even its definition in the case when $\Xi$ is not a single orbit is indirect and has to take maximum over all orbits in $\Xi$. In addition, we do not know whether the analogs of Theorem \ref{thm:main}\eqref{it:max},\eqref{it:mod} hold for it. The notion of symplectic complexity solves these issues, keeping the relation with representation theory.

\subsection{Main idea}
Our notion of symplectic complexity is based a notion of symplectic dimension that we define in \S \ref{sec:sd}. The symplectic dimension can be used in order to measure how far a subvariety of a  symplectic variety is from being isotropic. We define the $\Xi$-symplectic complexity of $\bfX$ to be the symplectic dimension of a certain subvariety of $T^*({\bf X \times G/B})$ attached to $\Xi$, where $\bf G/B$ denotes the flag variety of $\bf G$. 
}
\subsection{Conventions}

All the algebraic varieties and groups that we consider are defined over $\C$. We will identify them with their complex points. By a subvariety of an algebraic variety we will always mean a locally closed subvariety.

\subsection{Acknowledgements}
We thank  Edward Bierstone, Dmitry Kerner, and Michael Temkin for a fruitful e-mail correspondence.

Both authors were partially supported by ISF grant 249/17. A.A. was also supported by a Minerva Foundation grant. 

\section{Symplectic dimension}\label{sec:sd}

\begin{defn}
For a complex vector space $W$ and an anti-symmetric bilinear form $\omega$ on $W$, we define the symplectic dimension of $W$ by 
$\sdim_{\omega}(W):=(\dim W - \Ker \omega)/2$. We will drop the subindex $\omega$ if the form is understood. 
\end{defn}

\begin{defn}\label{def:sd}
For an algebraic variety $\bf Y$ and a  2-form $\omega$ on $\bf Y$, we define the symplectic dimension $\sdim_{\omega}\bfY$ of $\bfY$ as the maximum of the symplectic dimension of the tangent spaces to irreducible components  at their generic points.
\Dima{By convention, we define the symplectic dimension of an empty variety to be zero.} 
\end{defn}
\begin{prop}[Appendix \ref{sec:CG}]\label{prop:CG}
For any  algebraic variety $\bf Z$, a 2-form $\omega$ on $\bf Z$, and a (locally closed)\ subvariety $\bf Y\sub Z$, we have $\sdim \bfY \leq \sdim {\bf Z}$.
\end{prop}

\begin{cor}\label{cor:max}
Let $\bf Z$ be an algebraic variety and $\omega$ be a 2-form on $\bf Z$. Let ${\bf Z}_1, {\bf Z}_2$ be (locally closed) subvarieties  of $\bf Z$ such that ${\bf Z}={\bf Z}_1 \cup {\bf Z}_2$. Then $\sdim {\bf Z}=\max(\sdim {\bf Z}_1,\sdim{\bf Z}_2)$.
\end{cor}

\begin{lemma}\label{lem:simp}
For a smooth symplectic algebraic variety $\bf M$, and a (locally closed) subvariety $\bf Y\sub M$, we have $\sdim \bfY \geq \dim \bfY - \dim {\bf M}/2$. Furthermore, equality holds if and only if $\bf Y$ is coisotropic.
\end{lemma}
\begin{proof}
It is enough to prove the corresponding statements for linear spaces $V\sub W$, and a non-degenerate anti-symmetric form $\omega$. Let $V^\angle$ denote the space orthogonal to $V$ with respect to $\omega$. Then $\dim V^\angle=\dim W-\dim V$, and $\Ker(\omega|_V)=V \cap V^\angle\sub V^\angle$, and the equality $\Ker(\omega|_V)=V^\angle$ holds if and only if $V$ is coisotropic, i.e. $V\supset V^\angle$.  Thus
\begin{multline*}
\sdim_{\omega|_V}(V)=(\dim V - \Ker \omega|_V)/2\geq (\dim V - \dim V^\angle)/2=\\(\dim V - (\dim W - \dim V))/2=\dim V - \dim W/2,
\end{multline*}
and equality holds if and only if $V$ is coisotropic.
\end{proof}

\Dima{
\begin{prop}\label{prop:sympdom}
Let $\phi:{\bf Z\to Y}$ be a  morphism of algebraic varieties. Let $\omega$ be a 2-form on $\bfY$ and $\phi^*\omega$ be its pullback to $\bf Z$. Then
\begin{enumerate}[(i)]
\item \label{it:mapleq} $\sdim_{\phi^*\omega}(\bf Z)\leq \sdim_{\omega}(Y).$
\item \label{it:mapeq} If $\varphi$ is dominant then $\sdim_{\phi^*\omega}(\bf Z)=\sdim_{\omega}(Y).$
\end{enumerate}
\end{prop}
\begin{proof}
Without loss of generality we assume that $\bf Z$ is smooth.
\begin{enumerate}[{Step} 1.]
\item Statement \eqref{it:mapeq} in the case $\bf Y$ and $\bf Z$ are linear spaces, and $\phi$ is a linear operator.\\
This is straightforward from the definitions.
\item \label{Step:smooth}Statement \eqref{it:mapeq} in the case $\bf Y$ is smooth and the map $\phi$ is smooth. \\
This step immediately follows from the previous one.
\item Statement \eqref{it:mapleq} in the case $\bf Y$ is smooth and the map $\phi$ is smooth. \\
This step immediately follows from the previous one. 
\item \label{Step:leq} Statement \eqref{it:mapleq} in the general case.\\
We prove by Noetherian induction. Let $\bfY_1\subset \bfY$ be the singular locus of $\phi$ (which includes the singular locus of  $\bfY$). Let $\bfY_2:=\bfY\setminus \bfY_1$ denote the complement. Denote their preimages by ${\bf Z}_1:=\phi^{-1}(\bfY_1)$ and ${\bf Z}_2:=\phi^{-1}(\bfY_2)$. By the induction hypothesis, $\sdim {\bf Z}_1\leq\sdim \bfY_1$. By the previous step, $\sdim {\bf Z}_2\leq\sdim \bfY_2$. By Corollary \ref{cor:max}, we have
$$\sdim {\bf Z}=\max(\sdim {\bf Z}_1, \sdim {\bf Z}_2)\leq\max(\sdim {\bf Y}_1,\sdim {\bf Y}_2)=\sdim {\bf Y}$$

\item Statement \eqref{it:mapeq} for general dominant $\phi$.\\ 
By Step \eqref{Step:leq}, it is enough to show that $\sdim_{\phi^*\omega}(\bf Z)\geq\sdim_{\omega}(Y).$

 Let $\bfY_1,\bfY_2,{\bf Z}_1,{\bf Z}_2$ be as in the previous step.   By  Step \eqref{Step:smooth}, $\sdim {\bf Z}_2=\sdim \bfY_2$. Since ${\bf Z}_2$ is open in $\bf Z$,  we have $\sdim {\bf Z}\geq \sdim {\bf Z}_2$. Since $\bfY_2$ is open and dense in $\bf Y$,  we have $\sdim {\bf Y} = \sdim{\bf Y}_2$. Summarizing we have
$$\sdim {\bf Z}\geq\sdim {\bf Z}_2=\sdim {\bf Y}_2=\sdim {\bf Y}$$
\end{enumerate}
\end{proof}
}
\section{Preliminaries on modality and complexity}\label{subsec:ModCom}
\subsection{Modality}\label{subsec:mod}
\DimaB{Let us recall the notions of modality and generic modality from \cite{Vin} and \cite[\S 6]{Tim}.}
Let $\bfH$ be an algebraic group and $\bfX$ be a smooth $\bf H$-variety. 
By Rosenlicht's theorem \cite{Ros}, a non-empty open subset $\bf U\sub X$  has a geometric quotient by $\bfH$. In particular, all the orbits in $\bf U$ have the same dimension. This theorem implies the following equality.
$$\mathrm{tr}.\deg \C(\bfX)^{\bfH}=\min_{x\in \bfX(\C)}\mathrm{codim}_{\bfX}\bfH x$$
This number is called the \emph{generic modality} of $\bfX$, and denoted $d_{\bfH}(\bfX)$. 

Let $lcl(\bfX)$ denote the set of all irreducible $\bfH$-stable locally closed subvarieties, and $cl(\bfX)\sub lcl(\bfX)$ denote the subset consisting of closed subvarieties. 
Rosenlicht's theorem  implies 
\begin{align*}
\max_{\bfY\in lcl(\bfX)}d_{\bfH}(\bfY)&=\max_{\bfY\in cl(\bfX)}d_{\bfH}(\bfY)\\
&=\dim \{ (g,x)\in {\bf H\times X}\, \vert \, gx=c \}-\dim \bfH\\
&=\dim \{ (x,\xi)\in T^*{\bf  X}\, \vert \, \forall \alp \in Lie(\bfH), \, \langle \alp x,\xi\rangle =0\}-\dim \bfX
\end{align*}

This number is called the \emph{modality} of $\bfX$. We will denote it by $m_{\bfH}(\bfX)$.

\subsection{Complexity}\label{subsec:com}
Let $\bfG$ be a reductive algebraic group and $\bfX$ be a  smooth $\bf G$-variety. 
Let $\bf B\sub G$ be a Borel subgroup of $\bfG$. By \cite[Theorem 2]{Vin} (cf. \cite{Bri}) we have
$$m_{\bf B}(\bfX)=d_{\bf B}(\bfX)$$
This number is called \emph{the complexity of $\bfX$} and denoted by $c(\bfX)$. 

 Let $\fg$ denote the Lie algebra of $\bf G$, $\fg^*$ denote the dual space, and $\cN(\fg^*)\sub \fg^*$ denote the nilpotent cone.
For any point $x\in \bfX$, let $a_x:\bfG\to \bfX$ denote the action map, and $da_x:\fg\to T_x\bfX$ denote its differential.
We will denote by $\mu_{\bf G,X}$ the moment map $T^*\bfX\to \fg^*$  defined by
$$\mu_{\bf G,X}(x,\xi)(\alp):= \xi(da_x(\alp))$$

For a $\bfG$-invariant (locally closed) subvariety $\Xi \sub \cN(\fg^*)$, we defined in \cite[\S 2.2]{AG} the $\Xi$-complexity $c_{\Xi}(\bfX)$ by 
$$c_{\Xi}(\bfX):=\max_{\text{orbit } \bfO\sub \Xi}\left(\dim \mu^{-1}(\bfO)-\dim \bfO\right/2)-\dim \bfX$$
and proved that $c_{\cN}(\bfX)=c(\bfX)$.

\section{Symplectic complexity and the proof of Theorem \ref{thm:main}}\label{sec:sc}
Let $\bf G$ be a reductive algebraic group  and $\bf X$ be an algebraic $\bf G$-manifold.
Let $\bf B\sub G$ be a Borel subgroup, and let $$\kappa:=\kappa_{\bfX}:=\mu_{\bf \Delta G,X\times (G/B)}:T^*({\bf X\times G/B})\to \fg^*$$ denote the moment map under the diagonal action of $\bf G$.

\begin{notn}
Let $\nu:=\nu_{\bfX}:T^*({\bf X\times G/B})\to \fg^*$ be the composition of the projection $T^*({\bf X\times G/B})\onto T^*({\bf G/B})$ with the Springer resolution $\mu_{\bf G,G/B}:T^*({\bf G/B})\onto \cN(\fg^*)\sub \fg^*$.
\end{notn}

\begin{defn}
For a $\bfG$-invariant (locally closed) subvariety $\Xi \sub \cN(\fg^*)$, we define 
$${\bf T}_{\Xi}:=\kappa^{-1}(0)\cap \nu^{-1}(\Xi)\sub T^*({\bf X\times G/B})$$ and define the  $\Xi$-symplectic  complexity of $\bfX$ by  
 $$sc_{\Xi}(X)= \sdim({\bf T}_{\Xi}).$$
\end{defn}

Let us record the following immediate corollary of  Corollary \ref{cor:max}.

\begin{cor}\label{cor:orbitwice}
For any $\bfG$-invariant (locally closed) subvariety $\Xi \sub \cN(\fg^*)$, we have $$sc_{\Xi}(X)=\max_{\text{orbit } \bf O \sub \Xi}sc_{\bf O}(\bfX).$$
\end{cor}

\begin{prop}\label{prop:sd=c}
We have $sc_{\cN}(\bfX)=c(\bfX)$.
\end{prop}
\begin{proof}
We have
$$c(\bfX)=m_{\bf B}(\bfX)=m_{\bf G}(\bfX \times {\bf G/B})=\dim \kappa^{-1}(\{0\}) -\dim \bf X-\dim G/B$$
Since $\kappa^{-1}(0)$ is the union of conormal bundles to orbits, it is coisotropic. Thus Lemma \ref{lem:simp}  implies
$$\sdim \kappa^{-1}(0)=\dim \kappa^{-1}(0) -\dim T^*({\bf X\times G/B})/2 = \dim \kappa^{-1}(0) -\dim \bfX-\dim {\bf G/B}.$$
Since the image of $\nu$ lies in $\cN$, we have $\nu^{-1}(\cN)=T^*({\bf X\times G/P})$.
Altogether we have
$$sc_{\cN}(\bfX)=\sdim(\kappa^{-1}(0)\cap \nu^{-1}(\cN))=\sdim \kappa^{-1}(0)=\dim \kappa^{-1}(0) -\dim \bfX-\dim {\bf G/B}=c(\bfX).$$
\end{proof}



\begin{notn}\label{notn:Go}
For any nilpotent orbit $\bfO$ 
denote by $\Gamma_{\bfO}\subset T^*(\bfX)\times \bfO$  the image of ${\bf T}_{\bfO}$ under $$\Id \times \nu:T^*(\bfX)\times T^*({\bf G/B})  \to T^*(\bfX)\times \fg^*.$$ 
\end{notn}

Note that $\Gamma_{\bfO}$  is the intersection of  the graph of $-\mu_{\bf G, X}$ with the preimage of $\bfO$ under the projection $T^*(\bfX)\times \fg^*\onto \fg^*$. 

\begin{defn}\label{def:sympO}
Let $\bfO\sub  \fg^*$ be a coadjoint orbit, and let $a \in \bfO$. The Kirillov-Kostant-Souriau symplectic form on $\bfO$ is given at $a$ by $$\omega_{\bfO}(\ad^*(x)( a),\ad^*(y)(a))=\langle a, [x,y]\rangle.$$
Together with the standard symplectic form on $T^*\bfX$ this defines a symplectic form on $T^*\bfX\times \bfO$, and thus a 2-form on $\Gamma_{\bfO}$.
\end{defn}

\begin{lem}\label{lem:BGam}
For any nilpotent orbit $\bfO$ we have $\sdim(\Gamma_{\bfO})=sc_{\bfO}(\bfX)$.
\end{lem}
\begin{proof}
By \cite[Proposition 2.2.4(ii)]{AG}, the pullback under $\Id \times \nu$ of the form on $\Gamma_{\bfO}$ coincides with the form on $\mathbf{T}_\bfO$. Thus, by Proposition \ref{prop:sympdom}\eqref{it:mapeq}, we have $\mathbf{T}_\bfO=\sdim\Gamma_{\bfO}$.
\end{proof}

For a parabolic subgroup $\bf P\sub G$,  denote $\nu_P:=\mu_{\bf G,G/P}:T^*({\bf G/P})\to \overline{\bf O_P}\sub\cN(\fg^*)$ and 
$$\mu_{\bf P}:=\mu_{\bf \Delta G,X\times (G/P)}:T^*({\bf X\times G/P})\to \fg^*$$

Similar reasoning to the proof of Lemma \ref{lem:BGam} gives the following lemma. 

\begin{lem}\label{lem:PGam}
Let $\bf P\sub G$ be a parabolic subgroup, and $\bf O\sub \overline{\bf O_P}$ be a nilpotent orbit. Then $$\sdim(\Gamma_{\bfO})=\sdim(\mu_\bfP^{-1}(0)\cap \nu_{\bf P}^{-1}(\bfO))$$
\end{lem}

%







\begin{proof}[Proof of Theorem \ref{thm:main}]$\,$\\
\eqref{it:geq} By Corollary \ref{cor:orbitwice}, it is enough to prove for the case when $\Xi$ consists of a single orbit $\bfO$.  
 Since $\Gamma_{\bfO}$ is the graph of $-\mu_{\bf G,X}$ intersected with the preimage of $\bfO$ under the projection to $\fg^*$, we have $$\dim \Gamma_{\bfO}=\dim \mu^{-1}_{\bf G,X}(\bfO)= \dim \bfX+\dim \bfO/2+c_{\bfO}(\bfX)=\dim (T^*\bfX\times \bfO)/2+c_{\bfO}(\bfX).$$

 By  Lemmas \ref{lem:BGam} and \ref{lem:simp} we have 
$$sc_{\bf O}(\bfX)=\sdim\Gamma_{\bfO}\geq \dim \Gamma_{\bfO}-\dim (T^*\bfX\times \bfO)/2=c_{\bfO}(\bfX) $$

\eqref{it:max} By Corollary \ref{cor:orbitwice}, we can assume $l=1$ and thus $\Xi=\Xi_1$.  Let $ {\bf T}:={\bf T}_{\Xi}=\kappa^{-1}(0)\cap \nu^{-1}(\Xi)$  and ${\bf T}_i:={\bf T}_{{\bf Y}_i,\Xi}=\kappa_{\bfY_i}^{-1}(0)\cap \nu_{\bfY_i}^{-1}(\Xi)\subset T^*({\bf Y_i\times G/B})$. We have 
$$T^*\bfX \times T^*({\bf G/B})=\bigcup_i T^*\bfX|_{\bfY_i} \times T^*({\bf G/B})$$
Thus, by Corollary \ref{cor:max}, 
$$sc(\bfX)=\sdim {\bf T}=\max_i \sdim({\bf T}\cap T^*\bfX|_{\bfY_i} \times T^*({\bf G/B}))$$
We have a projection $p_i:T^*\bfX|_{\bfY_i} \to T^*\bfY_i$. It is easy to see that the pullback under $p_i$ of the symplectic form on $T^*\bfY_i$ is the restriction of the symplectic form on $T^*\bfX$. By definition, 
$(p_i\times \Id_{T^*({\bf G/B})})^{-1}({\bf T}_i)={\bf T}\cap T^*\bfX|_{\bfY_i} $ and thus $p_i\times \Id_{T^*({\bf G/B})}({\bf T}\cap T^*\bfX|_{\bfY_i} )={\bf T}_i$. By Proposition \ref{prop:sympdom}\eqref{it:mapeq}, this implies that $\sdim {\bf T}\cap T^*\bfX|_{\bfY_i} =\sdim {\bf T}_i$. Thus
$$sc_{\Xi}(\bfX)=\sdim {\bf T}=\max \sdim ({\bf T}\cap T^*\bfX|_{\bfY_i} )=\max \sdim {\bf T}_i=\max sc_{\Xi}(\bfY_i)$$

\eqref{it:mod} By Corollary \ref{cor:orbitwice}, and Lemmas \ref{lem:BGam} and  \ref{lem:PGam} we have 
$$sc_{\overline{\bf O_P}}(\bfX)=\max_{\bf O\ \sub \overline{O_P}}sc_{\bf O}(\bfX)=\max_{\bf O\ \sub \overline{O_P}} \sdim \Gamma_{\bf O}=\max_{\bf O\ \sub \overline{O_P}} \sdim (\mu_\bfP^{-1}(0)\cap \nu_{\bf P}^{-1}(\bfO))= \sdim (\mu_\bfP^{-1}(0))$$

\eqref{it:N} By Proposition \ref{prop:sd=c}
we have $sc_{\cN}(\bfX)=c(\bfX)$. By \cite[Proposition 2.2.8]{AG} we have $c_{\cN}(\bfX)=c(\bfX)$.
\end{proof}

\begin{proof}[Proof of Corollary \ref{cor:main}]
By Theorem \ref{thm:main}\eqref{it:mod}, \eqref{it:sc0} is equivalent to \eqref{it:Pfin}. By \cite[Theorem B]{AG}, \eqref{it:Pfin} is equivalent to \eqref{it:Ospher}.
\end{proof}
Theorem \ref{thm:main}\eqref{it:max} motivates us to consider the case when $\bf X$ is a homogeneous space, that is  $ \bf X= G/H$ for some algebraic subgroup $\bf H\sub G$.
 Let us give a formula for symplectic complexity in this case. 
Let $\fh$ denote the Lie algebra of $\bfH$, and let $\fh^{\bot}\sub \fg^*$ denote the space of functionals that vanish on $\fh$. 
\begin{prop}\label{prop:tran}
For any nilpotent orbit $\bfO\sub \cN(\fg^*)$ we have $sc_{\bfO}({\bf G/H})=\sdim(\bfO\cap \fh^{\bot})$.
\end{prop}
\begin{proof}
Let $\Gamma_{\bfO}\subset T^*({\bf G/H)\times \bfO}$ be as in Notation \ref{notn:Go}.
By Lemma \ref{lem:BGam}, we have $sc_{\bfO}({\bf G/H})=\sdim(\Gamma_{\bfO})$. Let $p:\Gamma_{\bfO}\onto {\bf G/H}$  be the natural projection, and let $\bfY:=p^{-1}([1])$. Then $\bfY$ is naturally isomorphic to $\bfO\cap \fh^{\bot}$. Let $a:{\bf G\times  Y\onto }{\Gamma}_{\bfO}$ denote the action map. Then by Proposition \ref{prop:sympdom}\eqref{it:mapeq}, $\sdim \Gamma_{\bfO}=\sdim_{a^*\omega}({\bf G\times  Y})$, where $\omega$   denotes the symplectic form on $\Gamma_{\bfO}$.

It is left to show that $\sdim_{a^*\omega}({\bf G\times  Y})=\sdim(\bfO\cap \fh^{\bot})$. For that purpose it is enough to show that for any smooth point $y\in \bfY$, $\sdim T_{(1,y)}{\bf G\times Y}=\sdim T_y\bfY$. For this  it is enough to show that $\fg$ lies in the radical of the  form $a^*\omega_{(1,y)}$, or equivalently, that $\fg \cdot y$ lies in the radical of $\omega$.
Since the action of $\bfG$ on $\Gamma_{\bfO}$ preserves $\omega$, the space $\fg \cdot y\sub T_y\Gamma_{\bfO}$ is isotropic.  On the other hand, $T_y\Gamma_{\bfO}=\fg\cdot y +T_y\bfY$. Thus it is enough to show that $\omega(\fg \cdot y,T_y\bfY)=0$. 

Let $\alp\in \fg$ and $\xi\in T_{y}\bfY$. Let $p_{\bfO}$ and $q$  be the projections from $\Gamma_{\bfO}$ to $\bfO$ and to $T^*({\bf G/H})$ respectively.
Let $\omega'$ and $\omega_{\bfO}$ denote the symplectic forms on $T^*({\bf G/H)}$ and $T^*\bfO$.  

Let us first show that $\omega'(\alp\cdot q(y),d_yq(\xi))= -\langle d_y p_{\bfO}(\xi), \alp\rangle $. Indeed, since $p(\bfY)=\{[1]\}$,
the vector $d_yq(\xi)$ is  ``vertical'', {\it i.e.} lies in $T_{q(y)} T_{[1]}^*{\bf(G/H)}\cong \fh^{\bot}\sub \fg^*$. From the definition of $\Gamma_{\bfO}$  we obtain that  $d_yq(\xi)=-d_yp_{\bfO}(\xi)$ as an element of $\fg^*$. From the definition of $\omega'$ (see {\it e.g.} \cite[Example 2.1.6]{AG}) we have $\omega'(\alp\cdot q(y),d_yq(\xi))= \langle d_yq(\xi), \alp\rangle $. Altogether we have 
$$\omega'(\alp\cdot q(y),d_yq(\xi))= \langle d_yq(\xi), \alp\rangle = -\langle d_yp_{\bfO}(\xi), \alp\rangle .$$

Let $\beta\in \fg$ be such that $\beta\cdot y=\xi$. Then we have 
\begin{align*}
&\omega(\alp\cdot y,\xi)= \omega'(d_yq(\alp\cdot y),d_yq(\xi))+
\omega_{\bfO}(d_yp_{\bfO}(\alp\cdot y),d_yp_{\bfO}(\xi))=\\
&\omega'(\alp\cdot q(y),d_yq(\xi))+
\omega_{\bfO}(\alp\cdot p_{\bfO}( y),\beta \cdot p_{\bfO}(y))= -\langle d_yp_{\bfO}(\xi), \alp\rangle +\langle p_{\bfO}(y), [\alp,\beta]\rangle=\\
&-\langle \beta\cdot p_{\bfO}(y), \alp\rangle +\langle p_{\bfO}(y), [\alp,\beta]\rangle
=\langle p_{\bfO}(y), [\beta,\alp]\rangle +\langle p_{\bfO}(y), [\alp,\beta]\rangle=0
\end{align*}
Thus $\omega(\fg \cdot y,T_y\bfY)=0$ and thus $\sdim_{a^*\omega}({\bf G\times  Y})=\sdim(\bfO\cap \fh^{\bot})$.
Summarizing we have
$$ sc_{\bfO}({\bf G/H})=\sdim(\Gamma_{\bfO})=\sdim_{a^*\omega}({\bf G\times  Y})=\sdim(\bfO\cap \fh^{\bot})$$
\end{proof}

%

\appendix

\section{Nash blowing up and the proof of Proposition \ref{prop:CG}}\label{sec:CG}

Let $M$ be a smooth variety.

\begin{notn}
For any vector space $V$ we denote by $Gr(V)$ the variety of all vector subspaces (of all dimensions) of $V$. This is a disconnected union of all Grassmanians on $V$. 

For any vector bundle $E$ on $M$ let $Gr(E)$ denote the Grassmanian bundle
$$Gr(E):=\{(x,L)\, \vert \,L\sub E_x \text{ a linear subspace}\}.$$

\end{notn}
\Dima{
\begin{defn}[Nash blowing up]
Let $Z\sub M$ be a locally closed subvariety, and $U\sub Z$ be an open subvariety of the smooth locus of $Z$. Let $p:Gr(TM)\to M$ denote the projection. 
Define the \emph{Nash blow up} of $Z$ (inside $M$) by
$$\tau_{U}(Z,M):=\overline{\{(x,T_xU)\, \vert \, x\in U\}}\cap p^{-1}(Z)\sub Gr(TM).$$
\end{defn}
}
\begin{remark}$\,$
\begin{enumerate}[(i)]
\item It is easy to see the variety $\tau_{U}(Z,M)$ does not depend on the choice of $U$. We will thus denote it by $\tau(Z,M)$.
\item One can show that  $\tau_{U}(Z,M)$ does not depend on the ambient manifold $M$ up to a canonical isomorphism, see \cite{Nob}.
\end{enumerate}
\end{remark}

The proof of Proposition \ref{prop:CG} is based on the following theorem.

\begin{thm}\label{thm:CG}
Let $Z_1\sub Z\sub M$ be \Dima{locally} closed subvarieties, and let $z\in Z_1$ be a smooth point. Then there exists $L$ such that $(z,L)\in  \tau(Z,M)$ and $T_zZ_1\sub L$.
\end{thm}
For the proof we will need the following ad-hoc definition and lemmas.

\begin{defn}
We call a triple $Z_1\sub Z\sub M$ of algebraic varieties as in Theorem \ref{thm:CG} \emph{good} if it satisfies the conclusion of the theorem. \end{defn}
\begin{lem}\label{lem:extL}
If $Z_1\sub Z\sub M$ is a good triple, then for any $(z,L)\in \tau(Z_1,M)$ there exists $L'$ such that $(z,L')\in \tau(Z,M)$ and $L\sub L'$. 
\end{lem}
\begin{proof}
Consider the closed subvariety 
$$C:=\{(z,L,L')\in Gr(M)\times_MGr(M)\, \vert \, \, (z,L)\in \tau(Z_1,M), (z,L')\in \tau(Z,M) \text{ and } L\sub L'\} $$
Denote by $p:Gr(M)\times_MGr(M) \to Gr(M)$ the projection on the first coordinate. We have to show that $p(C)\supset \tau(Z_1,M)$. The image $p(C)$ is closed since $C$ is closed and $p$ is proper. Thus it is enough to show that $p(C)$ includes a dense subset of $\tau(Z_1,M)$. Let $Z_1^0$ be the smooth locus of $Z_1$. The assumption that the triple $Z_1\sub Z\sub M$ is good implies that $p(C)$ includes $\tau(Z_1^0,M)$.\end{proof}

\begin{cor}\label{cor:4ple}
Let $Z_2\sub Z_1\sub Z\sub M$ be closed subvarieties. Suppose that the triples $Z_1\sub Z \sub M$ and $Z_2\sub Z_1\sub M$ are good. Then the triple $Z_2\sub Z \sub M$ is good. 
\end{cor}
\begin{proof}
Let $z\in Z_2$ be a smooth point. By the assumption that the triple $Z_2\sub Z_1\sub M$ is good, there exists $L$ such that $(z,L)\in \tau(Z_1,M)$ and $T_zZ_2\sub L$. By the assumption that the triple $Z_1\sub Z \sub M$ is good, Lemma \ref{lem:extL} implies that there exists $L'$ such that $(z,L')\in \tau(Z,M)$ and $L\sub L'$. Thus $T_zZ_2\sub L'$ as required.
\end{proof}

\begin{lem}\label{lem:irr} Let $Z_1\sub Z\sub M$ be closed subvarieties.  
If for any irreducible component $Z_1'\sub Z'$, the triple $Z'_1\sub Z\sub M$ is good then so is the triple $Z_1\sub Z\sub M$.
\end{lem}
\begin{proof}$\,$
This follows from the fact that the regular locus of $Z_1$ lies inside the union of the regular loci of $Z_1'$, where $Z_1'$ ranges over all irreducible components of $Z_1$. 
\end{proof}
%
%

\Dima{

The following lemma is obvious.
\begin{lemma}\label{lem:new}
Let $Z_1\sub Z\sub M$ be as in Theorem \ref{thm:CG}. Suppose that one of the following conditions holds.
\begin{enumerate}[(i)]
\item $Z$ is smooth.\label{it:Zsm}
\item There exists an open dense subset $U\sub Z_1$ such that the triple $U\sub Z\sub M$ is good.\label{it:U}
\end{enumerate}
Then the triple $Z_1\sub Z\sub M$ is good.
\end{lemma}

\begin{proof}[Proof of Theorem \ref{thm:CG}]
By Lemma \ref{lem:irr} we can assume without loss of generality that $Z_1$ is irreducible. By \cite{Whit1,Whit2}, $Z$ has a Whitney stratification $Z=\bigcup S_i$. Let $i_0$ be such that $S_{i_0}\cap Z_1$ is open \DimaA{and dense }in $Z_1$. By the definition of Whitney stratification the triple $S_{i_0}\sub Z \sub M$ is good. By Lemma \ref{lem:new}\eqref{it:Zsm}, the triple $S_{i_0}\cap Z_1\sub S_{i_0}\sub M$ is good. Thus, by Corollary \ref{cor:4ple},   the triple $S_{i_0}\cap Z_1\sub Z \sub M$ is good. Thus, by Lemma \ref{lem:new}\eqref{it:U}, the triple $Z_1\sub Z \sub M$ is good.
\end{proof}}
\subsection{Deduction of Proposition \ref{prop:CG}}

The following lemma is straightforward.
\begin{lemma}\label{lem:Linsdim}
Let $W$ be a linear space, and let $\Omega^2(W)$ denote the linear space of anti-symmetric bilinear forms on $W$. Then the function $\nu:Gr(W)\times \Omega^2(W)\to \bN$ given by $\nu(\omega,L):=\sdim(L,\omega|_L)$ is
\begin{enumerate}[(i)]
\item lower semicontinuous
\item \Dima{weakly} increasing with respect to inclusion  in $Gr(W)$.
\end{enumerate}
\end{lemma}
\Dima{
The following lemma well known.
\begin{lem}\label{lem:ExtForm}
Let $X\subset \mathbb{A}^n$ be a closed subvariety. Then any $k$-form on $X$ is the restriction of a $k$-form on  $\mathbb{A}^n$. \end{lem}
\begin{proof}
It is enough to prove for the case $k=1$, as a $k$-form is the wedge of 1-forms.
This case follows from the ``second exact sequence'', see {\it e.g.} \cite[\S Proposition II.8.4A]{Har} or \cite[Theorem 58]{Mat}.
\end{proof}
}
\begin{proof}[Proof of Proposition \ref{prop:CG}]
The statement is local, thus we can assume that $\bf Z$ is affine and embedded as a closed subvariety into an affine space $\bf M$. By Lemma \ref{lem:ExtForm} we can extend the form $\omega$ to $\bf M$. Identify the tangent space to $\bf M$ at all points and call it $W$. The form  $\omega$ on $M$ gives a map $M\to \Omega^2(W)$. 

Let $z\in {\bf Y}$ be a smooth point.  
By Theorem \ref{thm:CG} there exists $L$ such \Dima{that $(z,L)\in  \tau(Z,M)$ and $T_z\bfY\sub L$. By Lemma \ref{lem:Linsdim}(ii) we have 
 $\sdim T_z \bfY\leq \sdim L$. By Lemma \ref{lem:Linsdim}(i) we have 
 $\sdim L\leq \sdim {\bf Z}$.  }
\end{proof}

\end{document}